\documentclass{amsart}

\usepackage{graphicx,tikz,float,bm,amssymb}

\newtheorem{theorem}{Theorem}[section]

\theoremstyle{definition}

\theoremstyle{corollary}

\numberwithin{equation}{section}

\begin{document}

\title{Is the symmetric group Sperner?}

\author{Larry H. Harper}
\address{Department of Mathematics, University of California, Riverside, Riverside, CA 92521}
\email{harper@math.ucr.edu}

\author{Gene B. Kim}
\address{Department of Mathematics, University of Southern California, Los Angeles, CA 90089}
\email{genebkim@usc.edu}

%    General info

\date{December 21, 2018}


\begin{abstract}
An \emph{antichain} $\mathcal{A}$ in a poset $\mathcal{P}$ is a subset of $\mathcal{P}$ in which no two elements are comparable. Sperner showed that the maximal antichain in the Boolean lattice, $\mathcal{B}_n = \left\{ 0 < 1 \right\}^n$, is the largest rank (of size $\binom{n}{\lfloor n/2 \rfloor}$). This type of problem has been since generalized, and a graded poset $\mathcal{P}$ is said to be \emph{Sperner} if the largest rank of $\mathcal{P}$ is its maximal antichain. In this paper, we will show that the symmetric group $S_n$, partially ordered by refinement (or by absolute order equivalently), is Sperner.
\end{abstract}

\maketitle

%Section 1 - Introduction
\section{Introduction}

	A \emph{partial order}, $\leq$, on $S$, is a reflexive, antisymmetric, and transitive binary relation, and a \emph{poset}, $\mathcal{P} = \left( P, \leq \right)$, consists of a set, $P$, and a partial order $\leq$ on $P$. A \emph{chain} is a poset in which every pair of elements is comparable.

	The \emph{height} of $\mathcal{P}$, $h(\mathcal{P})$, is the maximum height of a chain in $\mathcal{P}$. The Jordan-Dedekind chain condition for $\mathcal{P}$ is that all maximal chains in an interval $[x,y] = \left\{ \, z \in \mathcal{P} : x \leq z \leq y \, \right\}$ have the same height. If $\mathcal{P}$ is connected and satisfies this condition, we can define a \emph{rank} function: select any $x_0 \in \mathcal{P}$ and define $r(x_0) = 0$. For any $x \neq x_0$, $r(x)$ is uniquely determined by $x \leq y \Rightarrow r(y) = r(x) + 1$. A \emph{graded poset} is a poset equipped with a rank function. We can define the \emph{levels} $N_i = \left\{ x \in P | r(x) = i \right\}$.

An \emph{antichain}, $A$, in $\mathcal{P}$ is a subset in which no two elements lie on a chain. Given a weighted poset, $\mathcal{P} = ( P, \leq, \omega )$, the \emph{width} of $\mathcal{P}$, $w$, is the maximum weight of an antichain in $\mathcal{P}$. If $\mathcal{P}$ is not explicitly weighted, the weight is implicitly the counting measure.

Given $\mathcal{P}$, \emph{Sperner's problem} is to find the width of $\mathcal{P}$. In~\cite{Sperner}, Sperner shows that the width of the (unweighted) Boolean lattice, $\mathcal{B}_n = \left\{ 0 < 1 \right\}^n$, is $\binom{n}{\lfloor n/2 \rfloor}$, the largest binomial coefficient. For $0 \leq k \leq h(\mathcal{P})$, we can also define a \emph{$k$-antichain}, $A_k$, in $\mathcal{P}$ to be a subset in which no $k+1$ elements lie on a chain. In~\cite{Erdos}, Erd{\"o}s extended Sperner's problem to finding the \emph{$k$-width}, $w_k(\mathcal{P}) = \max \left\{ w(A_k) \right\}$ and showed that
\[
	w_k(\mathcal{B}_n) = \sum_{j=1}^k \binom{n}{\lfloor (n+j-1)/2 \rfloor},
\]
the sum of the $k$~largest binomial coefficients. In~\cite{Stanley}, Stanley used techniques from algebraic geometry to show that Weyl groups, under Bruhat order, are Sperner. Engel wrote a book~\cite{Engel} which presents Sperner theory from a unified point of view, bringing combinatorial techniques together with methods from programming, linear algebra, probability theory, and enumerative combinatorics.

In \cite{Rota}, Rota conjectured that $\Pi_n$, the poset of partitions of $\left\{ 1, 2, \dots, n \right\}$, ordered by refinement, is Sperner. The conjecture was disproved by Canfield in \cite{Canfield} by using Graham-Harper reduction (\cite{GrahamHarper}) and probability theory. Canfield and Harper, in \cite{CanfieldHarper}, went further, showing that the ratio of the size of the largest antichain to the size of the largest rank goes to infinity. Canfield, in~\cite{Canfield2}, completed the resolution of Rota's question, showing the ratio of the size of the largest antichain in $\Pi_n$ and the largest Stirling number of the second kind (the rank sizes in $\Pi_n$) is $\Omega \left( n^\alpha \left( \ln n \right)^{-\alpha - \frac{1}{4}} \right)$, where $\alpha = \frac{2 - \ln 2}{4} \simeq \frac{1}{35}$. So, the ratio does go to infinity, but very slowly. In 1999, this result was designated one of ten outstanding results in order theory by the editor-in-chief of the journer \emph{Order}.

One of the natural questions that arises from Rota's conjecture is: what happens if we look at $S_n$, ordered by refinement? Given $\pi \in S_n$, we say that $\sigma$ is a \emph{refinement} of $\pi$ if we can take one of the cycles of $\pi$ and slice it into two. More formally, if $\pi = \pi_1 \pi_2 \cdots \pi_k$ is the cycle decomposition of $\pi$, for any two elements $i$ and $j$ on a cycle $\pi_m$, $\pi \cdot (i \, j)$ is a refinement of $\pi$. In this paper, we will take a category theoretical approach to show that $S_n$, ordered by refinement, is Sperner.

It is also worth mentioning another partial order, called absolute order, on $S_n$. The \emph{absolute length} of $\pi \in S_n$ is defined by
\[
	l_T(\pi) = n - \text{the number of cycles in $\pi$}.
\]
Then, the \emph{absolute order} on $S_n$ is defined by
\[
	\pi \leq_T \sigma \iff l_T(\sigma) = l_T(\pi) + l_T(\pi^{-1}\sigma).
\]
Armstrong, in~\cite{Armstrong}, showed that the absolute order is the reverse of refinement, and so, the main result of this paper implies that $S_n$, ordered by absolute order, is Sperner.

%%Section 2
\section{Flow morphisms}

In this section, we establish the groundwork to introduce the category $FLOW$. The objects of $FLOW$ are networks in the sense of Ford-Fulkerson~\cite{FordFulkerson}, and its morphisms preserve the Ford-Fulkerson flows (both underflows and overflows) on those networks.

A \emph{network} $N$ consists of an acyclic directed graph $G = (V,E)$ and a \emph{capacity function} $\nu: V \rightarrow \mathbb{R}^+$. For an edge $e \in E$, let $\partial_-(e)$ and $\partial_+(e)$ denote the head and tail of $e$, respectively. $V$ is partitioned into three sets, $R$, $S$, and $T$:
\begin{eqnarray*}
	S &= &\left\{ s \in V : \nexists e \in E, \partial_+(e) = s \right\}, \text{ called \emph{sources,}} \\
	T &= &\left\{ t \in V : \nexists e \in E, \partial_+(e) = t \right\}, \text{ called \emph{sinks, and}} \\
	R &= &V-S-T \text{, called \emph{intermediate vertices.}}
\end{eqnarray*}
An \emph{underflow} on $N$ is a function $f: E \rightarrow \mathbb{R}^+$ such that
\begin{itemize}
	\item for all $s \in S$, $\sum_{\partial_-(e) = s} f(e) \leq \nu(s)$,
	\item for all $t \in T$, $\sum_{\partial_+(e) = t} f(e) \leq \nu(t)$, and
	\item for all $r \in R$, $\sum_{\partial_-(e) = r} f(e) = \sum_{\partial_+(e) = r} f(e) \leq \nu(r)$.
\end{itemize}

An \emph{overflow} on $N$ is defined in the same way except that the inequalities are reversed. The quantity $net(f) = \sum_{s \in S} \sum_{\partial_-(e) = s} f(e)$ is the \emph{net S-T} flow of $f$, and the \emph{MaxFlow} of $N$ is defined as $MaxFlow(N) = \max_f net(f)$ over all underflows, $f$, on $N$. Similarly, \emph{MinFlow} of $N$ is defined as $MinFlow(N) = \min_f net(f)$ over all overflows, $f$, on $N$. By Ford-Fulkerson theory~\cite{FordFulkerson}, $MaxFlow(N) = MinCut(N)$, where a \emph{cut} is a set of vertices intersecting any path from a source to a sink. Also, $MinFlow(N) = MaxAntichain(N)$.

A \emph{bipartite network} is $V = S \cup T$ with all edges $e \in E$ directed from $S$ to $T$. A flow $f$ on a bipartite network $V = S \cup T$ is said to be a \emph{normalized flow} if
\begin{align*}
	\sum_{xy \in E} f(xy) = \frac{\omega(x)}{\omega(S)} &\qquad \text{for all $x \in S$, and} \\
	\sum_{xy \in E} f(xy) = \frac{\omega(y)}{\omega(T)} &\qquad \text{for all $y \in T$.}
\end{align*}
If $N$ is the Hasse diagram of a weighted and graded poset and every pair of consecutive ranks, $[N_k, N_{k+1}]$, accepts a normalized flow, then $N$ is said to have the \emph{normalized flow property (NFP)}.

For $G$ a bipartite graph with vertex sets $A$ and $B$, $G$ is said to satisfy Hall's matching condition, if for all $X \subseteq A$,
\[
	\lvert X \rvert \leq \lvert D(X) \rvert
\]
holds, where $D(X)$ is the set of vertices in $B$ connected to vertices in $X$. Sperner showed in his original problem that he only had to consider consecutive ranks at a time and if they satisfy Hall's condition, then the poset under consideration is Sperner.

When trying to prove Rota's conjecture, Graham and Harper came up with a strengthening of Hall's matching condition. A bipartite graph $G$ is said to satisfy \emph{normalized matching condition (NMC)} if for all $X \subseteq A$,
\[
	\frac{\lvert X \rvert}{\lvert A \rvert} \leq \frac{\lvert D(X) \rvert}{\lvert B \rvert}.
\]
The normalized matching condition is dual of the normalized flow property~\cite{FordFulkerson}. Harper has done extensive work in studying posets with NFP, and in~\cite{Harper}, he describes maps between these structures, called \emph{flow morphisms}. Let $M$ and $N$ be networks. Then, $\varphi: M \to N$ is a \emph{flow morphism} if
\begin{enumerate}
	\item $\varphi: G_M \to G_N$ is a graph epimorphism,
	\item $\varphi^{-1}(S_N) = S_M$ and $\varphi^{-1}(T_N) = T_M$,
	\item $\varphi$ is capacity preserving, i.e. for all $v \in N$, $\omega_M(\varphi^{-1}(v)) = \omega_N(v)$, and
	\item the preimage of every edge $e \in N$ has a normalized flow.
\end{enumerate}

This leads us to the category $FLOW$, whose objects are acyclic vertex-weighted networks and morphisms are precisely these flow morphisms. An important property of flow morphisms is that they preserve net $S$-$T$ flow, and so, $MaxFlow$ and $MinFlow$ problems on $M$ and $N$ are equivalent. In other words, if $M$ and $N$ are both in $FLOW$ and a flow morphism $\varphi$ exists between them, then the preimage of a maximum weight antichain of $N$ under $\varphi$ is a maximum weight antichain of $M$ (see~\cite{Harper} for a fuller discussion).

%%% Section 3
\section{$S_n$ is indeed Sperner}

In this section, we will prove that $S_n$ has normalized flow property which implies that $S_n$ is indeed Sperner.

\begin{theorem}
$S_n$ has normalized flow property.
\end{theorem}
\begin{proof}
We proceed by induction on $n$. The base case is trivial. As for the inductive step, let us assume that $S_n$ has normalized flow property.

The rank-weights of $S_n$, $\lvert S_{n,k} \rvert = s_{n,k}$, the Stirling numbers of the first kind, satisfy the recurrence relation
\[
	s_{n+1,k} = ns_{n,k} + s_{n,k-1}.
\]
Before continuing with the proof, we give an example of using this recurrence relation to view $S_4$ as four copies of $S_3$:

\begin{figure}[H]
\begin{center}
\scalebox{0.5}{
\begin{tikzpicture}
	\node (1a) at (4,0) {$(1\, 4\, 2\, 3)$};
	\node (1b) at (6,0) {$(1\, 4\, 3\, 2)$};
	\node (1c) at (12,0) {$(2\, 4\, 1\, 3)$};
	\node (1d) at (10,0) {$(2\, 4\, 3\, 1)$};
	\node (1e) at (16,0) {$(3\, 4\, 1\, 2)$};
	\node (1f) at (18,0) {$(3\, 4\, 2\, 1)$};
	\node (2b') at (4,6) {$(1\, 4\, 3)(2)$};
	\node (2b) at (6,6) {$(1\, 4\, 2)(3)$};
	\node (2a') at (2,6) {$(1\, 4)(2\, 3)$};
	\node (2d) at (8,6) {$(2\, 4\, 3)(1)$};
	\node (2e) at (12,6) {$(2\, 4\, 1)(3)$};
	\node (2e') at (10,6) {$(2\, 4)(1\, 3)$};
	\node (2g) at (14,6) {$(3\, 4\, 2)(1)$};
	\node (2h) at (16,6) {$(3\, 4\, 1)(2)$};
	\node (2i) at (18,6) {$(3\, 4)(1\, 2)$};
	\node (2j) at (22,6) {$(1\, 2\, 3)(4)$};
	\node (2k) at (24,6) {$(1\, 3\, 2)(4)$};
	\node (3a) at (2,12) {$(1\, 4)(2)(3)$};
	\node (3b) at (8,12) {$(2\, 4)(1)(3)$};
	\node (3c) at (14,12) {$(3\, 4)(1)(2)$};
	\node (3d) at (20,12) {$(2\, 3)(1)(4)$};
	\node (3e) at (22,12) {$(1\, 3)(2)(4)$};
	\node (3f) at (24,12) {$(1\, 2)(3)(4)$};
	\node (4) at (20,18) {$(1)(2)(3)(4)$};
	
	\draw[blue] (1a) -- (2b') (1a) -- (2b) (1a) -- (2a');
	\draw[blue] (1b) -- (2b') (1b) -- (2b) (1b) -- (2a');
	\draw[blue] (2b') -- (3a) (2b) -- (3a) (2a') -- (3a);
	
	\draw[blue] (1c) -- (2d) (1c) -- (2e) (1c) -- (2e');
	\draw[blue] (1d) -- (2d) (1d) -- (2e) (1d) -- (2e');
	\draw[blue] (2d) -- (3b) (2e) -- (3b) (2e') -- (3b);
	
	\draw[blue] (1e) -- (2g) (1e) -- (2h) (1e) -- (2i);
	\draw[blue] (1f) -- (2g) (1f) -- (2h) (1f) -- (2i);
	\draw[blue] (2g) -- (3c) (2h) -- (3c) (2i) -- (3c);
	
	\draw[blue] (2j) -- (3d) (2j) -- (3e) (2j) -- (3f);
	\draw[blue] (2k) -- (3d) (2k) -- (3e) (2k) -- (3f);
	\draw[blue] (3d) -- (4) (3e) -- (4) (3f) -- (4);
	
	\draw[red] (3a) -- (4) (3b) -- (4) (3c) -- (4);
	\draw[red] (2a') -- (3d) (2d) -- (3d) (2g) -- (3d);
	\draw[red] (2b') -- (3e) (2e') -- (3e) (2h) -- (3e);
	\draw[red] (2b) -- (3f) (2e) -- (3f) (2i) -- (3f);
	\draw[red] (1a) -- (2j) (1d) -- (2j) (1e) -- (2j);
	\draw[red] (1b) -- (2k) (1c) -- (2k) (1f) -- (2k);
	
	\draw[dashed, gray] (1a) -- (2g) (1b) -- (2d) (1c) -- (2h) (1d) -- (2b') (1e) -- (2e) (1f) -- (2b);
	\draw[dashed, gray] (2b') -- (3c) (2b) -- (3b) (2d) -- (3c) (2e) -- (3a) (2g) -- (3b) (2h) -- (3a);
	
\end{tikzpicture}
}
\end{center}
\caption{Viewing $S_4$ as four copies of $S_3$}
\end{figure}
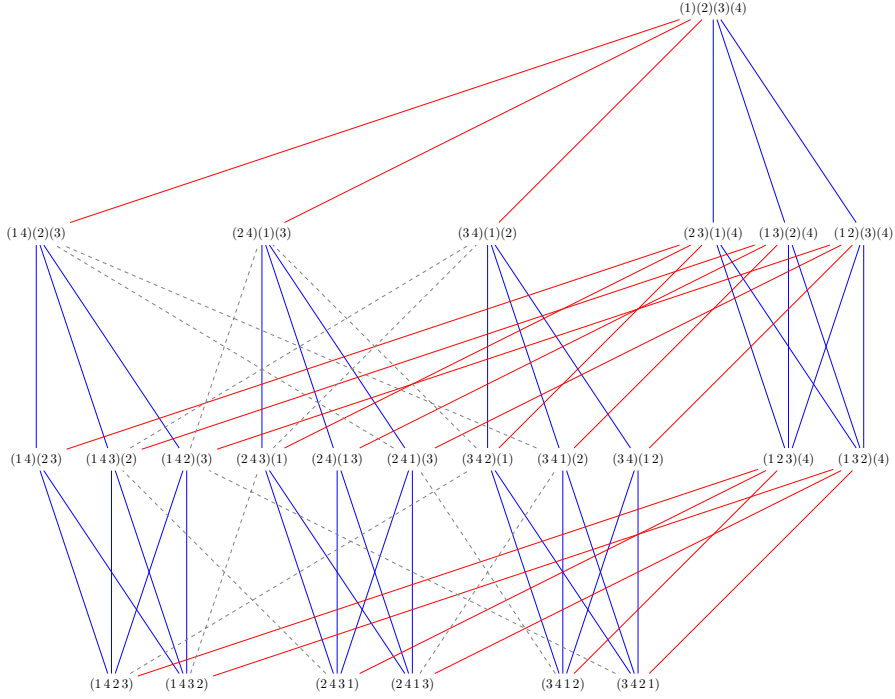

The copies of $S_3$ are arranged in a way that the first (blue) copy of $S_3$ has the six permutations $\pi$ with $\pi(1) = 4$, the second copy has $\pi$ with $\pi(2)=4$, the third copy has $\pi$ with $\pi(3) = 4$, and the fourth raised copy has $\pi$ with $\pi(4) = 4$. The red edges connect permutations from the raised copy to permutations of other copies, and the gray, dashed edges connect permutations from the lower copies to other lower copies.

A direct combinatorial proof of the recurrence follows from the observation that for $\pi \in S_{n+1,k}$, there are two possibilities:
\begin{enumerate}
	\item In the case that $\pi(n+1) = n+1$, we can remove $n+1$ from $\pi$ and have $\pi' \in S_{n,k-1}$. Conversely, adding a 1-cycle with $n+1$ to $\pi' \in S_{n,k-1}$ will give $\pi \in S_{n+1,k}$.
	\item In the case that $\pi(n+1) = i$, where $1 \leq i \leq n$, we can remove $n+1$ from the cycle containing $\pi(n+1) = i$ and define $\pi'(\pi^{-1}(n+1)) = i$, which will give $n$~copies, $S_{n,k}^{(i)}$ for $1 \leq i \leq n$. Conversely, the operation of defining $\pi \in S_{n+1,k}$ from $\pi' \in S_{n,k}$ can be done similarly.
\end{enumerate}
There is exactly one map $\pi \mapsto \pi'$ between the copy labeled $i$ and $n+1$, by construction, and the figure below is provided to help the reader visualize.

\begin{figure}[H]
\centering
\begin{tikzpicture}[scale=1.2]
	\draw [blue] plot [smooth cycle] coordinates {(0,0) (-1.2,2) (-2,4) (-0.8,2)};
	\draw [blue] plot [smooth cycle] coordinates {(1,0) (-0.2,2) (-1,4) (0.2,2)};
	\draw [blue] plot [smooth cycle] coordinates {(3,0) (1.8,2) (1,4) (2.2,2)};
	\draw [blue] plot [smooth cycle] coordinates {(4,0.5) (2.8,2.5) (2,4.5) (3.2,2.5)};
	\draw [blue] (1.2,0.5) node {.};
	\draw [blue] (1.7,0.5) node {.};
	\draw [blue] (2.2,0.5) node {.};
	\draw [black] (-1.2,2) -- (-0.8,2);
	\draw [black] (-0.2,2) -- (0.2,2);
	\draw [black] (1.8,2) -- (2.2,2);
	\draw [black] (2.8,2.5) -- (3.2,2.5);
	\draw [black] (-0.94,1.5) -- (-0.58,1.5);
	\draw [black] (0.06,1.5) -- (0.42,1.5);
	\draw [black] (2.06,1.5) -- (2.42,1.5);
	\draw [black] (3.06,2) -- (3.42,2);
	\draw [black] (-0.64,1) -- (-0.35,1);
	\draw [black] (0.36,1) -- (0.65,1);
	\draw [black] (2.36,1) -- (2.65,1);
	\draw [black] (3.36,1.5) -- (3.65,1.5);
	
	\draw [red] (-0.5,1) -- (3.5,1.5);
	\draw [red] (0.5,1) -- (3.5,1.5);
	\draw [red] (2.5,1) -- (3.5,1.5);
	
	\draw [red] (-0.76,1.5) -- (3.24,2);
	\draw [red] (0.24,1.5) -- (3.24,2);
	\draw [red] (2.24,1.5) -- (3.24,2);
	
	\draw [red] (-1,2) -- (3,2.5);
	\draw [red] (0,2) -- (3,2.5);
	\draw [red] (2,2) -- (3,2.5);
	
	\coordinate[label=below:$1$] (1) at (0,0);
	\coordinate[label=below:$2$] (2) at (1,0);
	\coordinate[label=below:$n$] (n) at (3,0);
	\coordinate[label=below:{$n+1$}] (n+1) at (4,0.5);
	
	\coordinate[label=left:$k$] (k) at (-0.94,1.5);
	\coordinate[label=left:{$k+1$}] (k+1) at (-1.2,2);
\end{tikzpicture}
\caption{Viewing $S_{n+1}$ in light of the recurrence relation}
\end{figure}
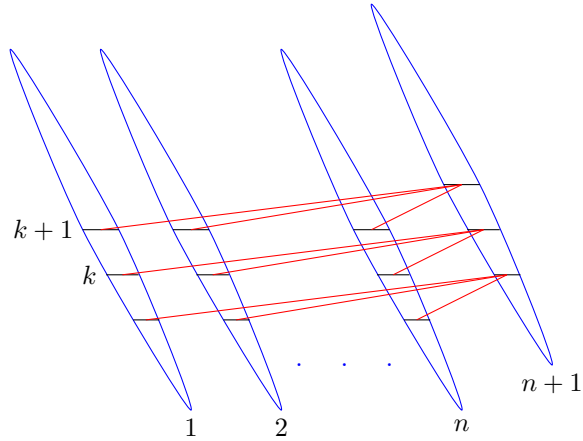

By the inductive hypothesis and the regularity between the $n$ blue copies of $S_n$, we can collapse the $n$ copies as in the figure below, where the collapsed copy is in bold.

\begin{figure}[H]
\centering
\begin{tikzpicture}
	\draw [blue, ultra thick] (0,0) -- (-2,4);
	\draw [blue] (1.5,0.5) -- (-0.5,4.5);
	
	\draw [red] (0,0) -- (1.5,0.5);
	\draw [red] (-0.25,0.5) -- (1.25,1);
	\draw [red] (-0.5,1) -- (1,1.5);
	\draw [red] (-0.75,1.5) -- (0.75,2);
	\draw [red] (-1,2) -- (0.5,2.5);
	\draw [red] (-1.25,2.5) -- (0.25,3);
	\draw [red] (-1.5,3) -- (0,3.5);
	\draw [red] (-1.75,3.5) -- (-0.25,4);
	\draw [red] (-2,4) -- (-0.5,4.5);
	
	\coordinate[label=below:{$[n]$}] (n) at (0,0);
	\coordinate[label=below:{$n+1$}] (n+1) at (1.5,0.5);
	\coordinate[label=left:$k$] (k) at (-1,2);
	\coordinate[label=left:{$k+1$}] (k+1) at (-1.25,2.5);
\end{tikzpicture}
\caption{``Collapsing'' the $n$ copies of $S_n$}
\end{figure}
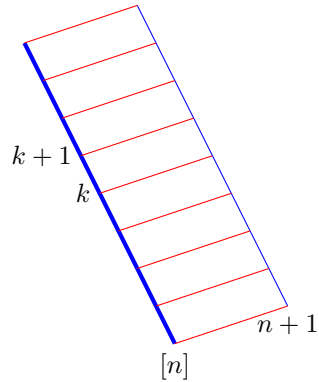

We claim that this new network satisfies the normalized matching condition. To show this, we consider the two consecutive ranks $k$ and $k+1$, which are shown with the corresponding vertex-weights:
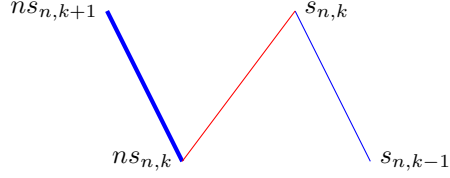
\begin{figure}[H]
\centering
\begin{tikzpicture}
	\draw [blue,ultra thick] (0,0) -- (-1,2);
	\draw [blue] (2.5,0) -- (1.5,2);
	\draw [red] (0,0) -- (1.5,2);
	
	\coordinate[label=left:{$ns_{n,k+1}$}] (A) at (-1,2);
	\coordinate[label=left:{$ns_{n,k}$}] (B) at (0,0);
	\coordinate[label=right:{$s_{n,k}$}] (C) at (1.5,2);
	\coordinate[label=right:{$s_{n,k-1}$}] (D) at (2.5,0);
\end{tikzpicture}
\caption{Two consecutive ranks $k$ and $k+1$}
\end{figure}

The only non-trivial equivalence class to show the normalized matching condition for is the class with the lower, right vertex. In other words, we need to show that
\[
	\frac{s_{n,k-1}}{s_{n,k-1} + ns_{n,k}} \leq \frac{s_{n,k}}{s_{n,k} + ns_{n,k+1}}.
\]
This is equivalent to $s_{n,k-1} s_{n,k+1} \leq s_{n,k}^2$, which is true due to the 2-positivity of $s_{n,k}$'s, which was proved in~\cite{Harper2}. Hence, NMC is satisfied, which in turn implies that $S_n$ satifies NMC, and so, has normalized flow property.
\end{proof}
\emph{Remark.} The lattice in Figure 3 is
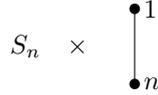
\begin{center}
	\begin{tikzpicture}
		\coordinate[label=right:$1$] (1) at (0,1);
		\coordinate[label=right:$n$] (n) at (0,0);
		\draw (1) -- (n);
		\fill (1) circle[radius=2pt];
		\fill (n) circle[radius=2pt];
		\coordinate[label=left:{$S_n \quad \times$}] (Sn) at (-0.5,0.5);
	\end{tikzpicture}
\end{center}
which has NFP by the Product theorem~\cite{Harper}. Our proof actually shows that
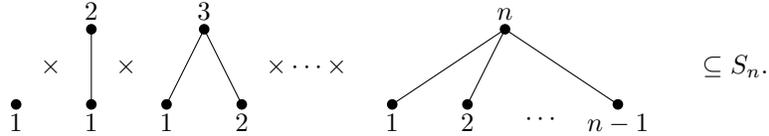
\begin{center}
	\begin{tikzpicture}
		\coordinate[label=below:$1$] (a1) at (0,0);
		\fill (a1) circle[radius=2pt];
		
		\coordinate[label=right:{$\times$}] (times1) at (0.2,0.5);
		
		\coordinate[label=below:$1$] (b1) at (1,0);
		\coordinate[label=above:$2$] (b2) at (1,1);
		\draw (b1) -- (b2);
		\fill (b1) circle[radius=2pt];
		\fill (b2) circle[radius=2pt];
		
		\coordinate[label=right:{$\times$}] (times2) at (1.2,0.5);
		
		\coordinate[label=below:$1$] (c1) at (2,0);
		\coordinate[label=below:$2$] (c2) at (3,0);
		\coordinate[label=above:$3$] (c3) at (2.5,1);
		\draw (c1) -- (c3) -- (c2);
		\fill (c1) circle[radius=2pt];
		\fill (c2) circle[radius=2pt];
		\fill (c3) circle[radius=2pt];
		
		\coordinate[label=right:{$\times \cdots \times$}] (times3) at (3.2,0.5);
		
		\coordinate[label=below:$1$] (d1) at (5,0);
		\coordinate[label=below:$2$] (d2) at (6,0);
		\coordinate[label=below:$\cdots$] (dd) at (7,0);
		\coordinate[label=below:{$n-1$}] (dn-1) at (8,0);
		\coordinate[label=above:$n$] (dn) at (6.5,1);
		\draw (d1) -- (dn) -- (d2);
		\draw (dn-1) -- (dn);
		\fill (d1) circle[radius=2pt];
		\fill (d2) circle[radius=2pt];
		\fill (dn) circle[radius=2pt];
		\fill (dn-1) circle[radius=2pt];
		
		\coordinate[label=right:{$\subseteq S_n$.}] (Sn) at (9,0.5);
	\end{tikzpicture}
\end{center}
Since the former has NFP by the Product theorem, the latter has NFP also.

Now that we have shown that $S_n$ has normalized flow property, we want to find a network we can map $S_n$ to, via a flow morphism, which is Sperner. In fact, we can collapse the network in Figure 3 further, just by keeping the same rank:

\begin{figure}[H]
\centering
\begin{tikzpicture}
	\draw [blue, ultra thick] (0,0) -- (-2,4);
	\draw [blue] (1.5,0.5) -- (-0.5,4.5);
	
	\draw [red] (0,0) -- (1.5,0.5);
	\draw [red] (-0.25,0.5) -- (1.25,1);
	\draw [red] (-0.5,1) -- (1,1.5);
	\draw [red] (-0.75,1.5) -- (0.75,2);
	\draw [red] (-1,2) -- (0.5,2.5);
	\draw [red] (-1.25,2.5) -- (0.25,3);
	\draw [red] (-1.5,3) -- (0,3.5);
	\draw [red] (-1.75,3.5) -- (-0.25,4);
	\draw [red] (-2,4) -- (-0.5,4.5);
	
	\coordinate[label=below:{$[n]$}] (n) at (0,0);
	\coordinate[label=below:{$n+1$}] (n+1) at (1.5,0.5);
	\coordinate[label=left:$k$] (k) at (-1,2);
	\coordinate[label=left:{$k+1$}] (k+1) at (-1.25,2.5);
	
	\draw (1.5,2) -- (3,2);
	\draw (2.85,2.15) -- (3,2) -- (2.85,1.85);
	
	\draw [blue,ultra thick] (5.25,0) -- (3,4.5);
\end{tikzpicture}
\caption{``Collapsing'' the $n$ copies of $S_n$}
\end{figure}
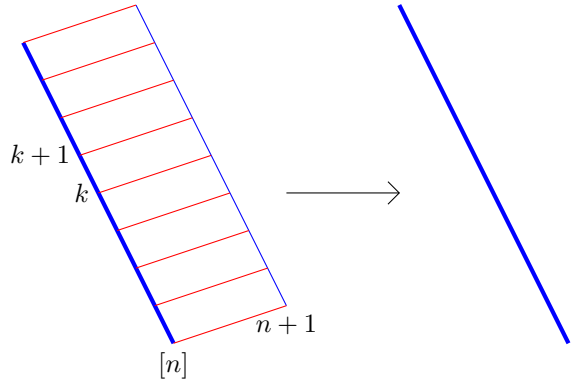

Since the resulting network is a totally ordered set, the largest antichain is going to be the rank/vertex, say $v$, with the largest vertex weight. The composition of the collapsings is a flow morphism, and so, the preimage of $v$ in $S_n$ will be the largest antichain. By construction, the preimage of each vertex in the totally ordered network is a rank in $S_n$, and so, the largest antichain in $S_n$ is the largest rank. Thus, $S_n$ is indeed Sperner.

\bibliographystyle{amsplain}

\end{document}